\documentclass{amsart}
\usepackage{amsmath, amssymb,epic,graphicx,mathrsfs,enumerate}
\usepackage[all]{xy}

\usepackage{amsthm}
\usepackage{amssymb}
\usepackage{latexsym}
\usepackage{longtable}
\usepackage{epsfig}
\usepackage{amsmath}
\usepackage{hhline}

 \DeclareMathOperator{\perm}{Sym}

 \DeclareMathOperator{\frat}{Frat}

\DeclareMathOperator{\md}{MaxDim}

\DeclareMathOperator{\alt}{Alt}
\DeclareMathOperator{\End}{End}

 \newcommand{\ol}{\overline}

\newcommand{\FF}{\mathbb F}

\renewcommand{\emptyset}{\varnothing}

\newtheorem{thm}{Theorem}%[section]
\newtheorem{cor}[thm]{Corollary}
 \newtheorem{lemma}[thm]{Lemma}
\newtheorem{prop}[thm]{Proposition} 
 \newtheorem{defn}[thm]{Definition}

\numberwithin{equation}{section}

\renewcommand{\footnote}{\endnote}
\newcommand{\ignore}[1]{}\makeglossary
\begin{document}
\bibliographystyle{amsplain}
\title[Maximal subgroups in general position]{Maximal subgroups of finite soluble groups\\ in general position}
\author{Eloisa Detomi and Andrea Lucchini}
\address{
Eloisa Detomi and Andrea Lucchini\\ Universit\`a degli Studi di Padova\\  Dipartimento di Matematica\\ Via Trieste 63, 35121 Padova, Italy}

\begin{abstract}For a finite group $G$ we investigate the difference between the maximum size $\md(G)$ of an \lq\lq independent\rq\rq\ family of maximal subgroups of $G$ and maximum size $m(G)$ of an irredundant sequence of generators of $G$. We prove that $\md(G)=m(G)$ if the derived subgroup of $G$ is nilpotent.
However $\md(G)-m(G)$ can be arbitrarily large: for any odd prime $p,$ we construct a finite soluble group with Fitting length 2 satisfying $m(G)=3$ and $\md(G)=p.$
\end{abstract}
\maketitle
\section{Introduction}
Let $G$ be a finite group. A sequence $(g_1,\dots,g_n)$ of elements of $G$ is said to be  {\emph{irredundant}} if $\langle g_j \mid j\neq i\rangle$ is properly contained in $\langle g_1,\dots,g_n\rangle$ for every $i\in \{1,\dots,n\}.$
 Let  $i(G)$ be the maximum size of any irredundant sequence in $G$ and let
 $m(G)$ be the maximum size of any  irredundant generating sequence of $G$
(i.e. an irredundant  sequence $(g_1,\dots,g_n)$ with the property that  $\langle g_1,\dots,g_n\rangle=G$).
 Clearly $m(G)\leq i(G)=\max\{m(H)\mid H\leq G\}.$
 The invariant $m(G)$ has received some attention (see, e.g., \cite{cc}, \cite{w}, \cite{sw}, \cite{ak}, \cite{al}, \cite{al2})
also because of its role in the efficiency of the product replacement algorithm \cite{pak}.
 In a recent paper Fernando \cite{rf} investigates  a natural connection between irredundant
generating sequences of $G$ and certain configurations of maximal subgroups of $G.$
A family of subgroups $H_i\leq G,$ indexed by a set $I$, is said to be in general position if
for every $i\in I,$  the intersection $\cap_{j\neq i}H_j$ properly contains $\cap_{j\in I}H_j.$
Define $\md(G)$ as the size of the largest family of maximal subgroups
of $G$ in general position. It can be easily seen that
$m(G)\leq \md(G)\leq i(G)$ (see, e.g., \cite[Proposition 2 and Proposition 3]{rf}).
However the difference  $\md(G)-m(G)$ can be arbitrarily large: for example if $G=\alt(5)\wr C_p$ is the wreath product
of the alternating group of degree 5 with a cyclic group of prime order $p,$ then $\md(G)\geq 2p$ but  $m(G)\leq 5$ \cite[Proposition 12]{rf}.
On the other hand  Fernando proves that $\md(G)=m(G)$ if $G$ is a finite supersoluble group \cite[Theorem 25]{rf},
  but gives also an example  of a finite soluble group $G$ with $m(G)\neq \md(G) $ \cite[Proposition 16]{rf}.

\

In this note we collect more information about the difference $\md(G)-m(G)$ when $G$ is a finite soluble group.
In this case $m(G)$ coincides with the number of complemented factors in a chief series of $G$ (see \cite[Theorem 2]{al}).
Our first result is that the equality $\md(G)=m(G)$ holds for a class of finite soluble groups, properly containing the class of finite supersoluble groups (see, e.g., \cite[7.2.13]{scott}).

\begin{thm}\label{thm:nilp} If $G$ is a finite group and the derived subgroup $G^\prime$ of $G$ is nilpotent, then $\md(G)=m(G).$
\end{thm}

However, already in the class of finite soluble groups with Fitting length equal to 2, examples can be exhibited of groups
$G$ for which the difference $\md(G)-m(G)$ is arbitrarily large.

\begin{thm}For any odd prime $p,$ there exists a finite group $G$ with Fitting length 2 such that
$m(G)=3$, $\md(G)=p$ and $i(G)=2p.$
\end{thm}

Notice that if $G$ is a soluble group with $m(G)\neq \md(G)$, then $m(G)\geq 3.$ Indeed if $m(G)\leq 2,$ then
a chief series of $G$ contains at most two complemented factors and it can be easily seen that this implies
that $G^\prime$ is nilpotent.

\section{Groups whose derived subgroup is nilpotent}
\begin{defn}
A family of subgroups $H_i\leq G,$ indexed by a set $I$, is said to be in general position if
for every $i\in I,$  the intersection $\cap_{j\neq i}H_j$ properly contains $\cap_{j\in I}H_j$
(equivalently, $H_i$ does not contain $\cap_{j\neq i}H_j$).
\end{defn}

Note that the subgroups $\{ H_i  \mid i \in I\}$ are in general position if and only,  whenever $I_1 \neq I_2$ are subsets of $S$, then $ \cap_{i \in I_1} H_i \neq \cap_{i \in I_2}H_i$ (see, e.g., Definition 1 in \cite{rf}).

\begin{lemma}\label{lem:inters}
Let $\FF$ be a field of characteristic $p$. Let $V$ a finite dimension $\FF$-vector space,
  let $H= \langle h \rangle$ where $h \in \FF^*$ such that $\FF=\FF_p[h]$ and set $G=V \rtimes H$.

If $M_1, \ldots , M_r$ is a set of maximal subgroups of $G$ supplementing $V$, then
$$  M_1 \cap \ldots \cap M_r= W \rtimes K$$
where $W$ is a $\FF$-subspace of $V$ and $K$ is either trivial or a conjugate $H^v$ of $H$, for some $v \in V$.
\end{lemma}

\begin{proof}
By induction on $r$ we can assume  that $T_1=M_1 \cap \ldots \cap M_{r-1}= W_1 \rtimes K_1$, where $W_1$ is a subspace of $V$ and
 $K_1=\{1\}$ or $K_1 = H^v$, $v \in V$. The maximal subgroup $M_r$ is a supplement of $V$, so we can write
 $M_r=W_2 \rtimes H^{w}$, where $W_2$ is a subspace of $V$ and $w \in V$.  For shortness, set $T_2=M_r$ and  $T=T_1 \cap T_2$.
Since $W_1$ and $W_2$ are normal Sylow $p$-subgroups of $T_1$ and $T_2$, respectively, their intersection $W=W_1 \cap W_2$ is a normal
Sylow $p$-subgroup of $T$. In the case where $T$ is not a $p$-group, then
 $T = W \rtimes K$ where $K$ is a non-trivial $p'$-subgroup of $T$. Then
  $K$ is contained in some conjugates $H^{v_1}$  and $H^{v_2}$ of the $p'$-Hall subgroups of $T_1$ and $T_2$, respectively.
% because all $p'$-Hall are conjugates in a  soluble group.
 In particular, there exists $1\neq y \in K$ such that $y=h_1^{v_1}=h_2^{v_2}$ for some $h_1, h_2 \in H$. It follows that $1\neq h_1=h_2 \in C_H(v_1-v_2)$.
From  $  C_H(v_1-v_2)\neq \{1\},$  we deduce that $v_1=v_2$.
Thus we have $T_1=W_1 \rtimes H^{v_1}$, $T_2=W_2 \rtimes H^{v_1}$ and $T=W \rtimes H^{v_1}$.
\end{proof}

\begin{cor}\label{cor:n+1}
In the hypotheses of Lemma \ref{lem:inters}, if  $M_1, \ldots , M_r$ are in general position, then
\begin{enumerate}
\item $r \le \dim (V)+1$;
\item if $r= \dim (V)+1$, then, for a suitable permutation of  the indices, $\bigcap_{i=1}^{r-1} M_i=H^v$ for some $v\in V$, and
$\bigcap_{i=1}^{r} M_i=\{1\}.$
\end{enumerate}
\end{cor}

\begin{proof}  Let $n=\dim V$.
Since the subgroups  $M_1, \ldots , M_r$ are in general position, the set of the intersections
 $T_j=\cap_{i=1}^{j} M_i$, for $j=1, \ldots , r$, is a strictly decreasing chain of subgroups.
 By Lemma \ref{lem:inters},  $T_i=W_i \rtimes K_i$, where $W_i$ is a $\FF$-subspace of $W_{i-1}$ and
  $K_i$ is either trivial or a conjugate of $H$.
 Note that  $n-1=\dim W_1 \ge \dim W_{i} \ge \dim W_{i+1} $. Moreover,  if $\dim W_{i} = \dim W_{i+1}$ for some index $i$, then $W_i=W_{i+1}$ and, since
 $T_i \neq T_{i+1}$, we have that  %,   there must exist an element  $  v\in V$ such that:
 \begin{itemize}
\item  $K_1,  \ldots ,K_{i}$ are non-trivial;
\item $K_{i+1}= \cdots = K_r=\{1\}$.
\end{itemize}
In particular there exists at most one index $i$ such that $\dim W_{i} = \dim W_{i+1}$. As $\dim W_1 = n-1$,
 it follows that we can have at most $n+1$ subgroups $T_i$, hence $r \le n+1$.

In the case where $r=n+1$, we actually have that  $\dim W_{i} = \dim W_{i+1}$ for at least one, and precisely one, index  $i$.
 This implies that $W_i=W_{i+1}$ and, setting
 $J=\{1, \ldots , n+1\}  \setminus \{i+1\}$ and $T=\cap_{l \in J}  M_{l}$, we get that
 $W_{n+1}$ coincides with the Sylow $p$-subgroup of $T$. Since $\dim W_{n+1}=0$ and $T\neq 1$
  we deduce that $T = H^v$, for some $v \in V$. Finally,  $T \cap M_{i+1}=\{1\}$.
\end{proof}

A proof of the following lemma is implicitly contained in Section 1 of \cite{rf},
  but, for the sake of completeness, we sketch a direct proof here.

\begin{lemma} \label{lem:abelian}
Let $H$ be an abelian finite group. The size of a set of subgroups in general position is at most $m(H)$.
\end{lemma}
\begin{proof} The proof is by induction on the order of $H$.
Let $\Omega=\{ A_1, \ldots , A_r\}$ be a set of subgroups  of  $H$ in general position. Without loss of generality we can assume that $\cap_{i=1}^rA_i=\{1\}$.
If $m=m(H)$, then $H$ decomposes as a direct product of $m$ cyclic  groups of prime-power order. Let $B$ be one of these factors, and let $X$ be the unique minimal normal subgroup of $B$. Since
 $\cap_{i=1}^rA_i=\{1\}$, there exists at least an integer $i$ such that $X$ is not contained in  $A_i$.
 It follows that $A_i \cap B=\{1\}$,
 hence $A_i \cong A_iB/B \le H/B$ and
 $$m(A_i) \le m(H/B) =m-1.$$
Now, the set of subgroups of $A_i$
$$\Omega^*=\{ A_j \cap A_i \mid j \neq i, \ 1 \le j \le r\}$$
%$$\Omega^*=\{ A_i \cap A_1, \ldots , A_i \cap A_r\}$$
 is in general position, hence, by inductive hypothesis, $|\Omega^*|=r-1 \le m(A_i)$.  Therefore $r \le m$.
\end{proof}

\begin{proof}[{\bf Proof of Theorem \ref{thm:nilp}}]
Since $$m(G)=m(G/\frat(G)) \text { and } \md(G)=\md(G/\frat(G)),$$ without loss of generality we can assume that $\frat(G)=1$. In this case the Fitting subgroup $F$ of $G$ is a direct product of minimal normal subgroups of $G$,
 it is abelian and complemented.
 Let $H$ be a complement of $F$ in $G$; note that, being  $G'$ nilpotent by assumption, $H$ is  abelian.   We can write $F$ as a product of $H$-irreducible modules
 $$F = V_1^{n_1} \times \cdots \times V_r^{n_r}$$
 where $V_1, \ldots , V_r$ are irreducible $H$-modules, pairwise not $H$-isomorphic.

 By \cite[Theorem 2]{al}
 $m(G)$ coincides with the number of complemented factors in a chief series of $G$, hence
 $$m(G)= \sum_{i=1}^r n_i +m(H).$$

 Let $\mathcal M$ be a family of maximal subgroups of $G$ in general position.

  Let $M_{0,1}, \ldots , M_{0,\nu_0}$ the  elements of $\mathcal M$  containing $F$. We can write
  $$M_{0,i} =F \rtimes Y_{i}$$
   where $Y_{i}$ is a maximal subgroup of $H$. Note that $Y_1, \ldots , Y_{\nu_0}$ are maximal subgroups of $H$ in general position, hence, by   Lemma \ref{lem:abelian}, $ \nu_0 \le \md(H) \le m(H)$.
% $ \nu_0 \le \md(H)= m(H)$, where the last equality holds because $H$ is abelian.

 If $M$ is a maximal subgroup supplementing $F$, then $M$ contains the subgroup $U_i = \prod_{j\neq i} V_j^{n_j}$ for some index $i$.
 In particular $M=(U_i \times W_i)\rtimes H^v$ for some
  $v\in V_i^{n_i}$ and some hyperplane $W_i$ of $V_i^{n_i}$.
 Set $C_i=C_H(V_i)$ and $H_i=H/C_i$. Then $\FF_i=\End_{H_i}(V_i)$ is a field and $V_i$ is an absolutely irreducible  $\FF_iH_i$-module.
% (see \cite[29.13]{curtis-reiner}
% \bibitem{curtis-reiner} Curtis, Charles W.; Reiner, Irving
%Representation theory of finite groups and associative algebras.
%Pure and Applied Mathematics, Vol. XI
  Since  $H_i$  is  abelian, $\dim _{\FF_i}V_i=1$, that is $V_i \cong \FF_i$,  and  hence $H_i$ is isomorphic to a subgroup of $\FF_i^*$
  generated by a primitive element.  In particular we can apply Corollary \ref{cor:n+1} to the group $V_i^{n_i} \rtimes H_i$.
Let $M_{i,1}, \ldots , M_{i, \nu_i}$ the maximal subgroups in $\mathcal M$ containing $U_i$; say
$$M_{i,l}=(U_i \times W_{i,l}) \rtimes H^{v_{i,l}},$$
 where $v_{i,l} \in V_i^{n_i}$.
Note that the subgroups $\ol M_{i,l}=W_{i,l} \rtimes H_i^{v_{i,l}}$, for $l \in \{1, \ldots , \nu_i\},$ are maximal subgroups of $V_i^{n_i} \rtimes H_i$ in general position, hence, by  Corollary \ref{cor:n+1},
 $$\nu_i \le n_i+1.$$

 If $\nu_i \le n_i$ for every $i \neq 0$, then
 $$ |\mathcal M| = \sum _{i=1}^{r} \nu_i +\nu_0 \le \sum _{i=1}^{r} n_i + m (H) =m(G), $$
and the result follows.

Otherwise let $J$ be the set of the integers $i \in \{1, \ldots ,r \}$ such that $\nu_i = n_i+1$.
 By Corollary \ref{cor:n+1}, we can assume that, for some $v_i \in V_i^{n_i}$,
 \begin{eqnarray*}
\bigcap_{l=1}^{n_i} M_{i,l}& =&U_i \rtimes H^{v_i},\\
\bigcap_{l=1}^{n_i+1} M_{i,l}& =&U_i \rtimes C_i.
\end{eqnarray*}
Recall that the $M_{0,j} =F \rtimes Y_{j}$, for $j=1 , \ldots , \nu_0$, are the elements of $\mathcal M$  containing $F$.
Our next task is to prove that
$$\Omega=\{ C_i \mid i\in J\} \cup \{ Y_{j} \mid j=1, \dots , \nu_0 \} $$
is a set of subgroups of $H$ in general position.

Assume, by contradiction, that for example $C_1 \ge  (\cap_{i \neq 1} C_i)  \cap (\cap_{ j =1}^{\nu_0} Y_{j})$;  then
$$M_{1, n_1+1} \ge U_1 \rtimes C_1 \ge (\cap_{l=1}^{n_1} M_{1,l}) \cap (\cap_{i \neq 1} (\cap_{ l=1}^{n_i+1} M_{i,l})) \cap  (\cap_{ j=1}^{\nu_0} M_{0,j})$$
 against the fact that $\mathcal M$ is in general position.
 Similarly, if %we assume that
  $Y_{1 } \ge  (\cap_{i \in J} C_i)  \cap (\cap_{ j \ne 1} Y_{j})$, % for some subsets $X$ and $Y$,
 then
 $$M_{0, 1} = F \rtimes Y_1 \ge (\cap_{i \in J}(\cap_{l=1}^{n_i+1} M_{i,l})) \cap  (\cap_{ j \neq 1} M_{0,j}),$$  a contradiction.

Now we can apply Lemma \ref{lem:abelian} to get  that $|\Omega| \le m(H)$. Therefore we conclude that
$$ |\mathcal M| = \sum _{i=1}^{r} \nu_i +\nu_0 \le  \sum _{i=1}^{r} n_i + |J| + \nu_0 = \sum _{i=1}^{r} n_i +|\Omega|
\le \sum _{i=1}^{r} n_i + m (H) =m(G), $$
and the proof is complete.
\end{proof}

\section{Finite soluble groups with $m(G)=3$ and $\md(G)\geq p$}
In this section we will assume that $p$ and $q$ are two primes
and that $p$ divides $q - 1.$
Let $\FF$ be the field with $q$ elements and let $C=\langle c \rangle $ be the subgroup of order $p$ of the multiplicative group of $\FF.$
Let $V=\FF^p$
be a $p$-dimensional vector space over $\FF$ and let $\sigma=(1,2,\dots,p)\in \perm(p).$  The wreath group $H=C\wr \langle \sigma \rangle$ has an irreducible action on $V$ defined as follows:
if $v=(f_1,\dots,f_p)\in V$ and $h=(c_1,\dots,c_p)\sigma \in H$, then $v^h=(f_{1\sigma^{-1}}c_{1\sigma^{-1}},\dots,f_{p\sigma^{-1}}c_{p\sigma^{-1}}).$
We will concentrate our attention on the semidirect product
$$G_{q,p}=V\rtimes H.$$

\begin{prop}\label{tre}$m(G_{q,p})=3.$
\end{prop}
\begin{proof}Since $V$ is a complemented chief factor of $G_{q,p},$ by \cite[Theorem 2]{al} we have  $m(G)=1+m(H)=1+m(H/\frat (H))=1+m(C_p\times C_p)=3.$
\end{proof}

\begin{prop}$i(G_{q,p})=2p.$
\end{prop}
\begin{proof}Let $B\cong C^p$ be the base subgroup of $H$ and consider $K=V\rtimes B\cong (\FF\rtimes C)^p.$
A composition series of $K$ has length $2p$ and all its factors are indeed complemented chief factors, so
$m(K)=2p.$ Now by definition $i(G_{q,p})=\max\{m(X)\mid X\leq G_{q,p}\}\ge m(K)= 2p.$
 On the other hand, $m(G)=3$ and,  if $X<G_{q,p},$ then $|X|$ divides $(pq)^p$ and the composition length of $X$ is at most $2p,$
so $m(X)\leq 2p$. Therefore $i(G_{q,p}) \le 2p,$  and consequently $i(G_{q,p})=m(K)=2p.$
\end{proof}

\begin{lemma}\label{big}$\md(G_{q,p})\geq p.$
\end{lemma}

\begin{proof} Let $e_1=(1,0,\dots,0), e_2=(0,1,\dots,0),\dots, e_p=(0,0,\dots,1) \in V$ and
let $h_1=(c,1,\dots,1), h_2=(1,c,\dots,1),\dots, h_p=(1,1,\dots,c) \in C^p\leq H.$
For any $1\leq i,j \leq p,$ we have
$$h_i^{e_j}=h_i \text { if } i\neq j, \quad h_i^{e_i}=((1/c-1)e_i)h_i.$$
But then, for each $i\in \{1,\dots,p\},$  we have
$$h_i \in \cap_{j\neq i}H^{e_j}, \quad h_i \notin H^{e_i},$$
hence $H^{e_1},\dots,H^{e_p}$ is a family of maximal subgroups of $G_{q,p}$ in general position.
\end{proof}

In order to compute the precise value of $\md(G_{q,p}),$ the following lemma is useful.

\begin{lemma}\label{delta}Let $v_1=(x_1,\dots,x_p)$ and $v_2=(y_1,\dots,y_p)$ be two different elements of $V=\FF^p$ and
let $\Delta(v_1,v_2)=\{i\in \{1,\dots,p\}\mid x_i=y_i\}.$  Then
\begin{itemize}
\item if  $|\Delta(v_1,v_2)|=0,$ then  $|H^{v_1}\cap H^{v_2}| \le p;$
\item  if $|\Delta(v_1,v_2)|=u\neq 0,$ then $|H^{v_1}\cap H^{v_2}|= p^u.$
\end{itemize}
\end{lemma}

\begin{proof}Clearly
 $|H^{v_1}\cap H^{v_2}|=|H\cap H^{v_2-v_1}|=|C_H(v_2-v_1)|.$
 If $\Delta(v_1,v_2)= \emptyset,$ then $C_H(v_2-v_1) \cap C^p=\{1\}$, hence $|C_H(v_2-v_1)| \le p.$ %and thus $|H^{v_1}\cap H^{v_2}|\le p.$
 If $|\Delta(v_1,v_2)|=u\neq 0,$ then
$$C_H(v_2-v_1)=\{(c_1,\dots,c_p) \in C^p \mid c_i=1 \text  { if } i \notin \Delta(v_1,v_2)\}\cong C^u$$ has order $p^u.$
\end{proof}

\begin{prop}If $p\neq 2,$ then $\md(G_{q,p})=p.$
\end{prop}
\begin{proof}By Lemma \ref{big} it suffices to prove that $\md(G_{q,p})\leq p.$ Assume that $\mathcal M$ is a family of maximal subgroups of $G=G_{q,p}$ in general position and let $t=|\mathcal M|.$
Let $M\in \mathcal M$. One of the following two possibilities occurs:
\begin{enumerate}
\item $M$ is a complement of $V$ in $G:$ hence $M=H^v$ for some $v\in V.$
\item $M$ contains $V:$ hence $M=V\rtimes X$ for some maximal subgroup $X$ of $H.$
\end{enumerate}
If $M_1$ and $M_2$ are two different maximal subgroups of  type (2), then $M_1\cap M_2= V\rtimes \frat(X)$ is contained in any
other maximal subgroup of type (2). Hence $\mathcal M$ cannot contain more then 2 maximal subgroups of type (2). Now we prove the following claim:
if $\mathcal M$ contains at least three different complements of $V$ in $G$, then $t\leq p.$ In order to prove this claim, assume, by contradiction
that $t>p$. This implies in particular that in the intersection $X$ of any two subgroups of $\mathcal M$, the subgroup lattice
 $\mathcal L (X)$  must contain a chain of length at least $p-1.$

% the subgroup lattice of the intersection of two subgroups of $\mathcal M$ must contain a chain of length at least $p-1.$
 Assume
that $H^{v_1}, H^{v_2}, H^{v_3}$ are different maximal subgroups in $\mathcal M.$ It is not restrictive to assume $v_1=(0,\dots,0)$. Let  $v_2=(x_1,\dots,x_p)$ and $v_3=(y_1,\dots,y_p).$
For $i\in \{2,3\}$, it must $|H\cap H^{v_i}|\geq p^{p-1}$,
hence, by Lemma \ref{delta}, $|\Delta(0,v_2)|=|\Delta(0,v_3)|=p-1,$ i.e. there exists $i_1 \neq i_2$ such that $x_{i_1}\neq 0$, $x_j=0$ if $j\neq i_1$,
$y_{i_2}\neq 0$, $y_j=0$ if $j\neq i_j$. But then $|\Delta(v_2,v_3)|=p-2$, hence $|H^{v_2}\cap H^{v_3}|=p^{p-2},$ a contradiction.
We have so proved that either $t\leq p$ or $\mathcal M$ contains at most 2 maximal subgroups of type (1) and at most
 2 maximal subgroups of type (2), and consequently $t\leq 4.$ It remains to exclude the possibility that $t=4$ and $p=3.$ By the previous considerations it is not restrictive to assume $\mathcal M=\{H, H^v, V\rtimes X_1, V\rtimes X_2\}$ where $X_1$ and $X_2$ are  maximal subgroups
of $H$ and $|\Delta(0,v)|=2.$ In particular we would have $H\cap H^v\leq C^3$: this excludes $C^3\in \{X_1,X_2\}$ but then
$X_1\cap C^3=X_2\cap C^3=\frat H=\{(c_1,c_2,c_3)\mid c_1c_2c_3=1\},$ hence $H\cap H^v \cap X_1= H\cap H^v\cap X_2,$ a contradiction.
\end{proof}

\begin{prop} $\md(G_{q,2})=3.$
\end{prop}
\begin{proof}By Lemma \ref{tre}, $\md(G_{q,2})\geq m(G_{q,2})=3.$ Assume now, by contradiction, that $M_1,M_2,M_3,M_4$ is a family of maximal
subgroups of $G_{q,2}.$ As in the proof of the previous proposition, at least two of these maximal subgroups, say $M_1$ and $M_2,$ are complements
of $V$ in $G_{q,2}.$ But then, by Lemma \ref{delta}, $|M_1\cap M_2| \le 2,$ hence $M_1\cap M_2\cap M_3=1,$ a contradiction.
\end{proof}

\end{document}